\documentclass[ECP]{ejpecp}

\usepackage[T1]{fontenc}
\usepackage[utf8]{inputenc}

\usepackage[numeric,initials,nobysame]{amsrefs}
\usepackage{chngcntr}

\SHORTTITLE{Bisection for kinetically constrained models}

\TITLE{Bisection for kinetically constrained models revisited\thanks{This work is supported by ERC Starting Grant 680275 ``MALIG.''}} 

\AUTHORS{%
  Ivailo~Hartarsky\footnote{CEREMADE, CNRS, Universit\'e Paris-Dauphine, PSL University, 75016 Paris, France
    \BEMAIL{hartarsky@ceremade.dauphine.fr}}}


\AMSSUBJ{60K35}
\AMSSUBJSECONDARY{82C22; 60J27}

\SUBMITTED{April 27, 2021} 
\ACCEPTED{October 16, 2021} 

\ARXIVID{2104.07883} 

\VOLUME{0}
\YEAR{2021}
\PAPERNUM{0}
\DOI{10.1214/YY-TN}

\ABSTRACT{The bisection method for kinetically constrained models (KCM) of Cancrini, Martinelli, Roberto and Toninelli is a vital technique applied also beyond KCM. In this note we present a new way of performing it, based on a novel two-block dynamics with a probabilistic proof instead of the original spectral one. We illustrate the method by very directly proving an upper bound on the relaxation time of KCM like the one for the East model in a strikingly general setting. Namely, we treat KCM on finite or infinite one-dimensional volumes, with any boundary condition, conditioned on any of the irreducible components of the state space, with arbitrary site-dependent state spaces and, most importantly, arbitrary inhomogeneous rules.
}

\newtheorem{obs}[theorem]{Observation}
\counterwithout{theorem}{section}
\counterwithout{equation}{section}

\newcommand{\cC}{\ensuremath{\mathcal C}}
\newcommand{\cD}{\ensuremath{\mathcal D}}
\newcommand{\cE}{\ensuremath{\mathcal E}}

\newcommand{\cI}{\ensuremath{\mathcal I}}

\newcommand{\cL}{\ensuremath{\mathcal L}}
\newcommand{\cM}{\ensuremath{\mathcal M}}

\newcommand{\cS}{\ensuremath{\mathcal S}}

\newcommand{\cU}{\ensuremath{\mathcal U}}

\newcommand{\bbN}{{\ensuremath{\mathbb N}} }

\newcommand{\bbP}{{\ensuremath{\mathbb P}} }

\newcommand{\bbR}{{\ensuremath{\mathbb R}} }

\newcommand{\bbZ}{{\ensuremath{\mathbb Z}} }

\newcommand{\D}{{\ensuremath{\Delta}}}

\newcommand{\g}{{\ensuremath{\gamma}}}
\newcommand{\G}{{\ensuremath{\Gamma}}}
\newcommand{\h}{{\ensuremath{\eta}}}

\renewcommand{\L}{{\ensuremath{\Lambda}}}
\newcommand{\m}{{\ensuremath{\mu}}}

\renewcommand{\o}{{\ensuremath{\omega}}}
\renewcommand{\O}{{\ensuremath{\Omega}}}
\newcommand{\p}{{\ensuremath{\pi}}}

\newcommand{\x}{{\ensuremath{\xi}}}

\newcommand{\z}{{\ensuremath{\zeta}}}

\newcommand{\var}{\operatorname{Var}}
\newcommand{\1}{{\ensuremath{\mathbb{1}}} }
\newcommand{\trel}{\ensuremath{T_{\mathrm{rel}}}}

\renewcommand{\le}{\leqslant}
\renewcommand{\ge}{\geqslant}
\renewcommand{\to}{\rightarrow}

\begin{document}
\section{Introduction}
\label{sec:intro}
The bisection method (also halving or two-block) is one of the fundamental techniques in the rigorous theory of kinetically constrained models (KCM), introduced by Cancrini, Martinelli, Roberto and Toninelli \cite{Cancrini08}*{Section 4} and inspired by \cite{Martinelli99}*{Proposition 3.5} for the Glauber dynamics of the Ising model. Its variations are frequently used for KCM \cites{Hartarsky20II,Hartarsky20FA,Hartarsky21a,Martinelli13}, but also successfully applied to other contexts \cites{Caputo12,Bhatnagar07}. The technique was originally developed to prove the positivity of the spectral gap of the East process (see also \cite{Aldous02}), as well as determining its sharp scaling at low temperature. For background on the East process we direct the reader to \cites{Faggionato13,Cancrini08,Ganguly15} and the references therein. An exposition of the original bisection method can be found in the upcoming monograph on KCM by Toninelli \cite{Toninelli21}.

In the present note we explore a new approach to the bisection method. In a way, the idea is the same, yet the proof and outcome are completely new. It is our hope that this new approach itself will be of independent interest and, in particular, our substitute for the two-block dynamics, Proposition~\ref{two-block}, and its proof. We apply it in the following setting of unprecedented generality. The (standard) terms used are defined formally in Section~\ref{sec:formal}. Consider KCM
\begin{itemize}
    \item on an arbitrary volume $L\subset\bbZ$, $1\le|L|\le\infty$, which need not be an interval;
    \item with arbitrary boundary conditions;
    \item conditioned to belong to an arbitrarily chosen irreducible component of the state space;
    \item with arbitrary on-site finite state spaces, which may vary from site to site and need not have uniformly bounded size or atom probabilities, but the probability of being infected is uniformly bounded from below by $q>0$;
    \item with arbitrary update rules, which may vary from site to site, but have a range uniformly bounded by $R<\infty$. Some sites may be completely unconstrained or, inversely, frozen.
\end{itemize}
In this setting we prove that for some $C_R>0$ depending only on $R$
\[\trel\le (2/q)^{C_R \log(\min((2/q),|L|))},\]
which is known to be sharp for all homogeneous rooted KCM on an interval \cites{Mareche20Duarte,Mareche20combi}, including East \cites{Aldous02,Cancrini08}, in the most interesting regime, $q\to0$. In addition, it may come as a surprise to specialists that this is also sharp for some homogeneous unrooted KCM on intervals, despite the fact that on $\bbZ$ their relaxation time is only $q^{-\Theta(1)}$  (see \cite{Martinelli19a} for definitions and background).

Let us note that for such general KCM there are usually many irreducible components (there are always at least two, save for trivialities) and their combinatorial structure can be very intricate. They have proved hard to deal with due to the long-range dependencies they introduce, like those present in conservative KCM. Consequently, the only nontrivial case in which the relaxation in an irreducible component is under control \cite{Blondel13} (see also \cites{Cancrini08,Cancrini09}) is the FA-1f model on an interval in its so-called ergodic component---the only nontrivial component for this KCM. An example of a situation in which such conditioned inhomogeneous one-dimensional KCM can arise naturally from ordinary KCM in higher dimensions can be found in \cite{Hartarsky20II}*{Appendix A.1} and originally motivated our work. 

We direct the reader to \cites{Shapira19polluted,Shapira20a,Shapira19} for inhomogeneous KCM, to \cites{Cancrini08,Cancrini09,Toninelli21} for KCM with various rules and boundary conditions and to \cites{Martinelli19a,Hartarsky20CBSEP} for general state spaces. Yet, we emphasise that no two among: general state spaces, inhomogeneous rules and irreducible components have featured simultaneously until present. Formally, as we will see, non-interval domains, boundary conditions and irreducible components other than the ergodic one can be absorbed in the inhomogeneity of the rules, but such arbitrarily inhomogeneous KCM have not been considered previously.

\section{Formal statement}
\label{sec:formal}
\subsection{Definition of the models and notation}
For all \emph{sites} $x\in \bbZ$ fix a finite positive probability space $(\cS_x,\p_x)$ called \emph{state space} and $\cI_x\subset\cS_x$ satisfying $\p_x(\cI_x)\ge q>0$. We say that $x\in \bbZ$ is \emph{infected} when the event $\cI_x$ occurs and \emph{healthy} otherwise. Thus, we refer to $q=\inf_{x\in\bbZ}\p_x(\cI_x)$ as the \emph{infection probability}. The \emph{volume} $L\subset\bbZ$ is a finite or infinite set. Consider the corresponding product space $\cS_L=\prod_{x\in L}\cS_x$ and measure $\p_L=\bigotimes_{x\in L}\p_x$. We will usually denote elements of $\cS_L$ (\emph{configurations}) by $\h,\o,\x$, etc.\ and corresponding restrictions to any $X\subset L$ by $\h_X$ and $\h_x$ when $X=\{x\}$. A \emph{boundary condition} is any $\o\in\cS_{\bbZ\setminus L}$ or an appropriate restriction, when some of the states of $\o$ are unimportant. Given two configurations $\h_L\in\cS_L$ and $\h'_{L'}\in\cS_{L'}$ for volumes $L,L'$ with $L\cap L'=\varnothing$, we denote by $\h_L\cdot\h'_{L'}\in\cS_{L\cup L'}$ the configuration equal to $\h_x$ if $x\in L$ and to $\h'_x$ if $x\in L'$.

For all $x\in L$ we fix \emph{an update family} $\cU_x$ that is a finite family of finite subsets of $\bbZ\setminus\{x\}$. Its elements are called \emph{update rules}. We assume that there exists a \emph{range} $R\in[1,\infty)$ such that for all $x\in L$, $U\in\cU_x$ and $y\in U$ we have $|x-y|\le R$. For a site $x\in L$, a configuration $\h\in\cS_L$ and a boundary condition $\o\in\cS_{\bbZ\setminus L}$, we say that the \emph{constraint at $x$ is satisfied} if
\[c_x^\o(\h)=\1_{\exists U\in\cU_x,\forall y\in U,(\h\cdot\o)_{y}\in\cI_y}\]
equals $1$. In words, we require that for at least one of the rules all its sites are infected, taking into account the boundary condition. The transitions allowed for the KCM are those changing the state of a single site whose constraint is satisfied (before and, equivalently, after the transition, since rules for $x$ do not contain $x$). In these terms, $\cU_x=\varnothing$ corresponds to a site unable to update under any circumstances, while $\cU_x\ni\varnothing$ corresponds to a site whose constraint is always satisfied. The transitions define an oriented graph with vertex set $\cS_L$ and symmetric edge set (containing the reverses of its edges). We call its connected components \emph{irreducible components} of the KCM and view them as events. Given an irreducible component $\cC\subset\cS_L$, we set $\m_L=\p_L(\cdot|\cC)$. We further write $\m_X=\m_L(\cdot|\h_{L\setminus X})$, $\m_x=\m_{\{x\}}$ for $x\in\bbZ$ and $X\subset \bbZ$ and denote by $\var_X$ and $\var_x$ the corresponding variances.

The \emph{general KCM} defined by $L$, $\cS_x$, $\p_x$, $\cI_x$, $\o$, $\cU_x$ and $\cC$ is the continuous time Markov process with generator and Dirichlet form acting on functions $f:\cC\to\bbR$ depending on the states of finitely many sites given by
\begin{align*}
\cL_L(f)(\h)={}&\sum_{x\in L}c_x^\o(\h)\cdot(\m_x(f(\h))-f(\h)),\\
\cD_L(f)={}&\sum_{x\in L}\m_L\left(c_x^\o\cdot\var_x(f)\right)
\end{align*}
respectively. In other words, this is the continuous time Markov process which resamples the state of each site at rate 1 w.r.t.\ $\m_x$, provided its constraint is satisfied. It is useful to note that when $c_x^\o=1$, we have $\m_x=\p_x$. For the existence of such infinite-volume processes see \cite{Liggett05} and for basic background refer to \cites{Cancrini08,Cancrini09}. It is also not hard to check that $\p_L$ and, therefore, $\m_L$ is a reversible invariant measure for the process. Finally, 
\begin{equation}
\label{eq:def:trel}
\trel^{-1}=\inf_{f\neq\text{const.}}\frac{\cD_L(f)}{\var_L(f)}\in[0,1]
\end{equation}
is the \emph{spectral gap} of $\cL_L$ or inverse \emph{relaxation time}.

\subsection{Result}
With this terminology, our main result is stated as follows.
\begin{theorem}
\label{main}
There exists an absolute constant $C>0$ such that for any range $R\in[1,\infty)$, infection probability $q\in(0,1]$, volume $L\subset \bbZ$ and general KCM with these parameters it holds that
\begin{equation}
\label{eq:main}
\trel\le (2/q)^{CR^2\min(\log|L|,R\log(2/q))}.
\end{equation}
\end{theorem}
\begin{remark}
Equation~\eqref{eq:main} and its proof apply to general KCM on a circle $\bbZ/n\bbZ$ (uniformly on $n$). For trees of maximum degree $\D$ and diameter $D$ we can only retrieve that for some $C$ depending on $\D$ and $R$,
\[\trel\le (2/q)^{C\log D}.\]
\end{remark}

Before moving on to the proof of Theorem~\ref{main}, let us mention a few applications.

As noted in Section~\ref{sec:intro}, Theorem~\ref{main} is sharp not only for all homogeneous rooted supercritical models, but also for some unrooted ones. Indeed, an unrooted KCM in finite volume may lack clusters of infections mobile in both directions, but only be able to create them, using ones mobile in a single direction. Such is the case of the homogeneous $\{\{-2\},\{1,2\}\}$-KCM on $L=\{1,\dots,2n\}$ with healthy boundary condition, only $1$ and $|L|$ infected initially (so that it is in its ``ergodic component,'' able to infect the entire volume). As usual, a test function showing that $\trel\ge \exp(\O(\log^2(1/q)))$ for $|L|=1/q\to\infty$ is the indicator of configurations reachable from the initial state above without creating $\log(1/q)/10$ infections simultaneously. 

This phenomenon is not related to the lack of symmetry---the same reasoning applies to the $\{\{-9,-8,-6\},\{-7,-6,-4\},\{-6,-5,-3\},\{3,5,6\},\{4,6,7\},\{6,8,9\}\}$-KCM on $L=\{1,\dots,6n+3\}$ with the ergodic initial condition $\{1,2,4,6n-1,6n,6n+3\}$. Indeed, for this KCM the sites $\{6n-1,6n,6n+3\}$ are unable to infect anything, while $\{1,2,4\}$ may infect a group of sites of the form $6k+1,6k+2,6k+4$, provided the previous such group is already present to its left. Hence, a similar test function yields the optimality of Eq.~\eqref{eq:main}.

A more subtle application of Theorem~\ref{main} concerns homogeneous KCM in higher dimensions. Consider a one-dimensional subset $L$ of $\bbZ^d$ for $d\ge 2$, that is a sequence of sites such that if two sites are at distance more than $C$ in the sequence, they are at distance more than $R$ in $\bbZ^d$, where $R$ is the range of the $d$-dimensional KCM and $C<\infty$ is some constant possibly depending on the KCM. In words, this is a discrete version of a one-dimensional manifold: a parametrised curve which may not come close to itself non-locally in the parametrisation (e.g.\ a line segment intersected with $\bbZ^d$). Fixing the state $\o$ of all sites in $\bbZ^d\setminus L$, the dynamics allowed to flip only sites in $L$ becomes a one-dimensional general KCM (with range $C$ rather than $R$, but finite) treated by our result. Notice that, even if the original $d$-dimensional KCM is homogeneous and considered in infinite volume on its ergodic component, the resulting one-dimensional one may become inhomogeneous due to $L$ bending in $\bbR^d$ or due to $\o$ not being translation invariant. Furthermore, it may naturally occur that this restricted dynamics is no longer able to infect all of $L$ without changing the boundary condition $\o$ (which is prohibited), so irreducible components become relevant.

The above application is the main motivation for our work. Indeed, control on the relaxation of a line segment at the boundary of a large infected region with arbitrary boundary condition elsewhere was needed for establishing refined universality results for two-dimensional KCM in \cite{Hartarsky20II}. More precisely, \cite{Hartarsky20II}*{Lemma A.1} is a direct corollary of Theorem~\ref{main} providing a much more tractable proof than the cumbersome canonical path approach outlined in that work.

\section{Proof}
\label{sec:proof}
Let us begin with a straightforward but important corollary of reversibility.
\begin{obs}
\label{bootstrap}
The irreducible component of a general KCM naturally identifies with the set of sites which can be eventually updated, together with the state of all remaining sites. We call the set of the sites that can be updated in $L$ \emph{closure} and denote it by $\{\h\}_L^\o\subset L$. We denote the state of the remaining sites by $\h^0:=\h_{L\setminus\{\h\}_L^\o}$ and refer to it as \emph{initial condition}.
\end{obs}
Since sites in $L\setminus\{\h\}_L^\o$ can never be updated, we may remove them from $L$ and replace $\o$ by $\o\cdot\h_{L\setminus\{\h\}_L^\o}$. With this reduction, we may assume that $\{\h\}_L^\o=L$ for the original general KCM. Further note that we may absorb any boundary condition in the inhomogeneous update rules by removing infected sites in $\o$ from update rules and removing update rules containing non-infected sites in $\o$. Thus, we may further assume that our initial general KCM is defined so that its rules do not depend on the boundary condition and therefore discard $\o$. Moreover, once the boundary condition is irrelevant, we may replace $L$ by an interval of length $|L|$, if $L$ is finite, and $\bbN$ or $\bbZ$, if $L$ is infinite in one or two directions. Indeed, we can map $L$ onto $\{1,\dots, |L|\}$, $\bbN$, $-\bbN$ or $\bbZ$, preserving the order, and this does not increase the range $R$. Finally, approximating $L$ by large finite segments if it is infinite (see \cite{Cancrini08}*{Section 2} and \cite{Liggett05}*{Chapter 4}), we may assume $|L|<\infty$.

Henceforth, we fix a general KCM specified by its volume $L$, state spaces $(\cS_x,\p_x)$, infection events $\cI_x$, and update families $\cU_x$ subject to the above simplification (we call such a model \emph{simplified}):
\begin{itemize}
\item for all $x\in L$ and $U\in \cU_x$ we have $U\subset L$;
\item $\{\h\}_L=L$;
\item $L=\{1,\dots,|L|\}$ with $|L|<\infty$.
\end{itemize}
Note that in the course of the proof we will consider domains smaller than $L$ and will then specify the closure, initial and boundary conditions. We will prove Theorem~\ref{main} by induction on $|L|$. The induction step is provided by the following two-block result, which is the core of the argument.
\begin{proposition}
\label{two-block}
Let $L_1=\{1,\dots,\ell\}$ and $L_2=\{\ell-\D+1,\dots,|L|\}$ with $\ell\in[1,|L|]$ and $\D\in[0,\ell]$. Then
\begin{equation}
\label{eq:two-block}
\var_L(f)\le \g(\D) \sum_{i\in\{1,2\}}\m_L\left(\var_{L_i}\left(f|\{\h\}^{\h_{L\setminus L_{i}}}_{L_i},\h_{L_i}^0\right)\right),\end{equation}
setting for some absolute constant $C>0$
\[\g(\D)=\begin{cases}1+\exp\left(-\D q^{CR}/\left(CR^2\right)\right)&\D\ge CR^2/q^{CR}\\
(2/q)^{CR^2}&\text{otherwise}.\end{cases}\]
\end{proposition}
\begin{remark}
\label{aux}
This statement can be viewed as a Poincar\'e inequality for a Markov process with two symmetric moves performed at rate 1. We update the state $\h_{L_i}$ from the measure $\p_{L_i}$ conditioned on the irreducible component of the current state in $L_i$. Crucially, the closure is taken only inside $L_i$, without infecting sites in $L\setminus L_{i}$ and going back to $L_i$, but using $\h_{L\setminus L_{i}}$ as a (frozen) boundary condition. In particular, the variance in Eq.~\eqref{eq:two-block} is not $\var_{L_i}(f)$.
\end{remark}

Before proving Proposition~\ref{two-block}, let us briefly recall how to deduce Theorem~\ref{main}, referring to \cite{Cancrini08}*{Theorem 6.1} for more details. Let $\G_l$ denote the maximum of $\trel$ over all general KCM (simplified or not) of range at most $R$ and infection probability at least $q$ on volume with cardinal at most $l$. Plugging Eq.~\eqref{eq:def:trel} into the r.h.s.\ of Eq.~\eqref{eq:two-block}, we get
\[\var_L(f)\le \g(\D)\G_{\max(L_1,L_2)}\m_L\left(\sum_{x\in L} c_x\cdot\var_x(f)+\sum_{x\in L_1\cap L_2}c_x\cdot\var_x(f)\right).\]
We average this over $N\approx|L|^{1/3}$ choices of $\ell$, so that the $L_1\cap L_2$ for different choices are disjoint. All $\ell$ are chosen so that $\ell-|L|/2\in[-N\D/2,N\D/2]$ for $\D\approx|L|^{1/3}$ fixed. This yields the recurrence relation
\[\G_{|L|}\le (1+1/N)\g(\D)\G_{|L|/2+N\D},\]
since the simplification operation may only decrease $|L|$ and $R$ and increase $q$.  Iterating this inequality, we derive the desired Eq.~\eqref{eq:main}.

Thus, our task is to prove Proposition~\ref{two-block}, for which we need the following.
\begin{claim}
\label{proba}
Let $\L$ be a volume. Then for any irreducible component $\cC=(\{\h\}_\L^\o,\h^0)$, under $\p_\L(\cdot|\cC)$ the infections in the closure $(\1_{\cI_x})_{x\in\{\h\}_\L^\o}$ stochastically dominate i.i.d.\ Bernoulli variables with parameter $q$.
\end{claim}
\begin{proof}
Fix $x\in\{\h\}_\L^\o$ and a configuration $\h\in\cC$. Observe that if $\h_x\not\in\cI_x$, then every $\h'$ such that $\h'_y=\h_y$ for all $y\in \L\setminus\{x\}$ is also in $\cC$ by Observation~\ref{bootstrap}. Hence, conditionally on $\h_{\L\setminus \{x\}}$ and $\cC$, either $\cI_x$ occurs a.s.\ or $\h_x$ has the law $\p_x$. In both cases the conditional law of $\1_{\cI_x}$ dominates a Bernoulli one with parameter $q$.
\end{proof}
\begin{proof}[Sketch of the easier case of Proposition~\ref{two-block}]
As a warm-up, let us sketch the proof of Eq.~\eqref{eq:two-block} with $\g(\D)=(2/q)^{CR^2}$, which is valid for all values of $\D$. 

We aim to couple two copies $\h$ and $\h'$ of the chain in Remark~\ref{aux}, so that they meet with appropriate rate. To do this, we require that the following sequence of events all occur in both chains uninterrupted by any other updates. Each chain is updated on $L_1$ to a state such that the sites in $L_1\setminus L_2$ at distance at most $R$ from $L_2$ (if $|L_1\setminus L_2|<R$, take all sites in $L_1\setminus L_2$) which are in the closure $\{\h\}_{L_1}^{\h_{L\setminus L_1}}$ of the current state in $L_1$ are infected. Then do the same in $L_2$, infecting all possible sites at distance at most $R$ from $L_1$ in $L_2\setminus L_1$. Repeat this couple of operations $R+1$ times. The configurations provided to $\h$ and $\h'$ so far are chosen independently, but updates occur at the same times for both. Next update $L_1$ in both $\h$ and $\h'$ to the same configuration still with infections next to $L_2$ as above and finally update $L_2$ in both chains to the same configuration, forcing them to meet.

In order for this to work, we need two ingredients. Firstly, we need to check that the rate at which this sequence of updates occurs is at least $(q/2)^{CR^2}$. Indeed, the probability that fewer than $2(R+2)$ updates occur up to time $2^{CR}$ is small; the probability that the first $2(R+2)$ updates occur in the right positions (in $L_1$ then in $L_2$, again in $L_1$, etc.) is $2^{-2(R+2)}$; from Claim~\ref{proba} the probability of infecting the desired (at most $R$) sites is at least $q^{R}$ (this needs to happen $4R+5$ times in total). Secondly, we need to check that this is a valid coupling, namely that in the last two steps the two chains are indeed resampled from the same distribution. For this it suffices to see that after $R+1$ repetitions of the alternating updates in $L_1$ and $L_2$, necessarily the $R$ sites in $L_2\setminus L_1$ closest to $L_1$ are all infected. This is not surprising, since each time we provide the best possible boundary condition and so the sequence of these boundary conditions is nondecreasing. 

Therefore, it remains to see that after a couple of updates as above either the boundary condition is already fully infected or it increases strictly. Assume the last $R$ sites in $L_1\setminus L_2$ remain unchanged after updating $L_2$ and then $L_1$ as above. Then none of the remaining non-infected sites could be updated at all, since even the best boundary condition $L_2$ can provide does not allow $L_1$ to infect them. Since it was assumed that $\{\h\}_L=L$, this implies that all $R$ sites are infected, as desired.
\end{proof}
Note that the above is sufficient to obtain Theorem~\ref{main} for $|L|\le (2/q)^{CR}$.

\begin{proof}[Proof of the harder case of Proposition~\ref{two-block}]
We consider two copies $(\h(t))_{t\ge 0}$, $(\h'(t))_{t\ge 0}$ of the process from Remark~\ref{aux}. It is well known \cite{Levin09}*{Proposition~4.7, Corollary~12.6, Remark~13.13}\footnote{For continuous time Markov chains the spectral radius in \cite{Levin09}*{Corollary~12.6} is replaced by $e^{-1/\trel}$.} that it suffices to couple them so that the probability that they do not meet before time $T$ is at most $C e^{-T/\g(\D)}$ for any $T$ large enough. Observe that whenever several successive updates are performed at $L_1$ (and similarly for $L_2$), the final result is preserved if we discard all but the last update, since the dynamics of Remark~\ref{aux} is of Glauber type. Hence, we may consider a discrete time chain with the same state space which updates $L_1$ at odd steps and $L_2$ at even ones (so the update from time $0$ to time $1$ is in $L_1$). Conditionally on the number of alternating updates $N$ up to time $T$, after removing redundant ones as indicated above, the two chains $\h$ and $\h'$ meet if their discrete time versions do. We denote the latter by $\o$ and $\o'$. 

We assume that $\D\ge CR^2/q^{CR}$, the alternative being treated in a similar but simpler way as sketched above. We call any set $B\subset L_1\cap L_2$ of $2R+1$ consecutive sites a \emph{block} and say it is \emph{infected} if $\cI_x$ occurs for all $x\in B$. 
\begin{claim}
\label{isolation}
Fix $\theta\in\cS_L$ such that $\{\theta\}_L=L$ and an infected block $B=x+\{1,\dots,2R+1\}$. Then $\{\theta\}_{\{1,\dots,x\}}^{\theta_B}=\{1,\dots,x\}$.
\end{claim}
\begin{proof}
This follows immediately from the fact that the closure is increasing in the set of infections (since constraints are), since an infected block is the maximal possible boundary condition.
\end{proof}

Let us denote by $M=\ell-\lfloor\D/2\rfloor+\{-R,\dots,R\}$ the middle block. Our coupling of $\o$ and $\o'$ is the following for integer $t\ge 0$.
\begin{itemize}
    \item The two chains evolve independently between $2t$ and $2t+2$, unless \begin{equation}
    \label{eq:condition}\left\{\o(2t)\right\}_{L_1}^{(\o(2t))_{L\setminus L_1}}\cap\left \{\o'(2t)\right\}_{L_1}^{(\o'(2t))_{L\setminus L_1}}\supset M.
    \end{equation}
    \item If Eq.~\eqref{eq:condition} occurs, we first sample two independent configurations $\x,\x'\in\cS_{L_1}$ with the laws of $(\o(2t+1))_{L_1}$ and $(\o'(2t+1))_{L_1}$, given $\o(2t)$ and $\o'(2t)$. Let $x+\{-R,\dots,R\}$ be the rightmost block contained in $L_1\cap L_2$ infected in both $\x$ and $\x'$, if it exists. We set $\o(2t+1)=\x\cdot(\o(2t))_{L\setminus L_1}$ and
    \[
    \o'(2t+1)=\x_{\{1,\dots,x\}}\cdot\x'_{\{x+1,\dots,\ell\}}\cdot(\o'(2t))_{L\setminus L_1}
    \]
    and sample $\o(2t+2)=\o'(2t+2)$ with their (common) law given the state at time $2t+1$. If no such block exists, $(\o(2t+1))_{L_1}=\x$ and $(\o(2t+1))_{L_1}=\x'$ and the two evolve independently between $2t+1$ and $2t+2$.
\end{itemize}
This is a legitimate Markov coupling of the homogeneous chains $(\o(2t))_{t\ge 0}$ and $(\o'(2t))_{t\ge 0}$. Indeed, by Claim~\ref{isolation},  conditionally on $x+\{-R,\dots,R\}$ being the rightmost infected block, $\x_{\{1,\dots,x\}}$ and $\x'_{\{1,\dots,x\}}$ are identically distributed.
We define $X(t)=\left|M\cap \{\o(2t)\}_{L_1}^{(\o(2t))_{L\setminus L_1}}\right|$ and similarly for $\o'$. Equation~\eqref{eq:condition} then reads $X(t)=X'(t)=2R+1$. We will lower bound $\min(X(t),X'(t))$ by the discrete time Markov chain $Y(t)$ on $\{0,\dots,2R+2\}$ which:
\begin{itemize}
\item starts at $0$;
\item is absorbed when reaching $2R+2$;
\item increments by $1$ with probability
\begin{equation}
\label{eq:increment:probability}\left(1-\left(1-q^{4R+2}\right)^{\D/(4R+3)}\right)^4;
\end{equation}
\item jumps to $0$ with the complementary probability.
\end{itemize}
We call a transition of $Y$ to $0$ a \emph{failure}.

\begin{lemma}\label{failure}
For all $t\ge0$, $\bbP(Y(t)\neq 2R+2)\ge \bbP(\o\text{ and $\o'$ have not met by time } 2t)$.
\end{lemma}
\begin{proof}
It suffices to prove that if Eq.~\eqref{eq:condition} holds, $\o$ and $\o'$ meet in two steps at least with the probability in Eq.~\eqref{eq:increment:probability}, while if Eq.~\eqref{eq:condition} fails, at least with the probability in Eq.~\eqref{eq:increment:probability} each of $X$ and $X'$ not equal to $2R+1$ increases. 

Assume that $X(t)=X'(t)=2R+1$. By Claim~\ref{isolation} (note that if $X(t)=2R+1$, then $M$ can be infected inside $L_1$) we have that $\{(\o(2t))_{L_1}\}_{L_1}^{(\o(2t))_{L\setminus L_1}}\supset \{1,\dots,\ell-\lceil \D/2\rceil\}$. Recalling Claim~\ref{proba} and the fact that the configurations $\x$ and $\x'$ are chosen independently, we obtain that the probability that $\o(2t+2)\neq\o'(2t+2)$ is at most $(1-q^{4R+2})^{\D/(4R+3)}\le\eqref{eq:increment:probability}$, since $\D\ge CR^2/q^{CR}$.

Next assume that $\min(X(t),X'(t))<2R+1$. Then $\o(2t+2)$ and $\o'(2t+2)$ are independent conditionally on $\o(2t),\o'(2t)$, so it suffices to establish that
\begin{equation}
\label{eq:Xincrement}
\bbP\left(X(t+1)\ge \min(X(t)+1,2R+1)|\o(2t)=\h\right)\ge \left(1-\left(1-q^{2R+1}\right)^{\D/(4R+3)}\right)^2
\end{equation}
for any $\h$ compatible with $X(t)$. Consider the event $\cE$ that in $\o(2t+1)$ for at least one block $B$ to the left of $M$ all sites in $B\cap\{\o(2t)\}_{L_1}^{(\o(2t))_{L\setminus L_1}}$ are infected and likewise for $\o(2t+2)$, a block $B'$ to the right of $M$ and $B'\cap\{\o(2t+1)\}_{L_2}^{(\o(2t+1))_{L\setminus L_2}}$. By Claim~\ref{proba}, $\bbP(\cE|\o(2t)=\h)$ is bounded by the r.h.s.\ of Eq.~\eqref{eq:Xincrement}. Thus, Proposition~\ref{two-step} below concludes the proof of Lemma~\ref{failure}.
\end{proof}
\begin{lemma}
\label{two-step}
In the above setting $\cE$ implies $X(t+1)\ge\min(X(t)+1,2R+1)$.
\end{lemma}
\begin{proof}
Fix blocks $B=x+\{-R,\dots,R\}$ and $B'=x'+\{-R,\dots,R\}$ witnessing the occurrence of $\cE$ and denote $\theta=(\o(2t+1))_{L_1}$, $\z=(\o(2t+1))_{L\setminus L_1}$, $\theta'=(\o(2t+2))_{L_2}$ and $\z'=(\o(2t+2))_{L\setminus L_2}$ for shortness.

We know that $\{\theta\}_{L_1}^{\z}\cap B$ is infected. Therefore, 
\begin{align}
\nonumber\{\theta\}_{\{x,\dots,\ell\}}^{\theta_{\{x-R,\dots,x-1\}}\cdot\z}&{}=\{\theta\}_{L_1}^\z\cap\{x,\dots,\ell\},\\
\label{eq:isolation}\{\theta\}_{\{1,\dots,x\}}^{\theta_{\{x+1,\dots,x+R\}}}&{}=\{\theta\}_{L_1}^\z\cap\{1,\dots,x\},
\end{align}
by Claim~\ref{isolation} applied to the general KCM restricted to $L_1$ after performing the simplifications from the beginning of Section~\ref{sec:proof}. Consequently, 
\begin{align}
\nonumber\cM:={}&M\cap \{\o(2t)\}_{L_1}^\z=M\cap\{\theta\}_{L_1}^\z=M\cap\{\theta\}_{\{x,\dots,\ell\}}^{\theta_{\{x-R,\dots,x-1\}}\cdot\z}\\
\subset{}&M\cap\left\{\theta_{L_1\cap L_2}\cdot\z\right\}_{L_2}^{\z'}=M\cap\{\o(2t+1)\}_{L_2}^{\z'}.\label{eq:monotonicity}
\end{align}
Using the analogous relation for the second transition, we obtain $X(t+1)\ge X(t)$ and equality holds iff Eq.~\eqref{eq:monotonicity} and its analogue are equalities. 

Assume that $X(t+1)=X(t)$. Then, for an augmented configuration $\bar\o$ equal to $\o(2t+1)$ with additionally all sites in $\cM$ infected, neither update can modify states in $M\setminus\cM$. Thus, for $\bar\o$ the block $M$ simultaneously has the isolation property Eq.~\eqref{eq:isolation} of $B$ and its analogue for $B'$. Hence, 
\[\{\bar\o\}_L =  \{\bar\o\}_{\{1,\dots,\ell-\lfloor\D/2\rfloor-R-1\}}^{\bar\o_M}\cup\cM\cup \{\bar\o\}_{\{\ell-\lfloor\D/2\rfloor+R+1,\dots,|L|\}}^{\bar\o_M},\]
since the update rules of each site in $M$ cannot look both to the left of $M$ and to its right. Recalling that $\{\bar\o\}_L\supset\{\o(2t+1)\}_L=L$, we get $\cM=M$ yielding the desired conclusion that $X(t+1)=X(t)=2R+1$, since $X(t)=|\cM|$.
\end{proof}
Returning to the proof of Proposition~\ref{two-block}, clearly, in order for $Y$ not to be absorbed, at least one failure must occur in every $2R+2$ steps. Hence, the probability that $\h$ and $\h'$ have not met by time $T\ge 2$ is at most
\begin{multline*}e^{-T}\sum_{n=0}^\infty\frac{T^n}{n!}\left(1-\left(1-\left(1-q^{4R+2}\right)^{\D/(4R+3)}\right)^{8R+8}\right)^{\lfloor n/(4R+4)\rfloor}\\
\le e^{-T}T^{9R}+\exp\left(-T\left(1-A^{1/(9R)}\right)\right),
\end{multline*}
since $N$ has the Poisson distribution with parameter $T$, setting
\[A=1-\left(1-\left(1-q^{4R+2}\right)^{\D/(4R+3)}\right)^{8R+8}\le (8R+8)\exp\left(\frac{-\D q^{4R+2}}{4R+3}\right).\qedhere\]
\end{proof}



\providecommand{\bysame}{\leavevmode\hbox to3em{\hrulefill}\thinspace}
\providecommand{\MR}{\relax\ifhmode\unskip\space\fi MR }

\providecommand{\MRhref}[2]{%
  \href{http://www.ams.org/mathscinet-getitem?mr=#1}{#2}
}
\providecommand{\href}[2]{#2}

\ACKNO{We thank Cristina Toninelli for stimulating discussions and helpful remarks. We are also grateful to the anonymous referee for careful proofreading and helpful comments on the presentation of the paper.}
\end{document}